%% file: main.tex
\documentclass[reqno]{amsart}

\usepackage[utf8]{inputenc}
\usepackage{amssymb}
\usepackage{amsthm}
\usepackage{amsmath}
\usepackage{tikz}

\usepackage[toc,page]{appendix}

\usepackage{enumitem}

\usepackage{tikz-cd}

\usepackage{tikz-cd}

\usepackage{tikz}
\usetikzlibrary{datavisualization}
\usetikzlibrary{datavisualization.formats.functions}

\usetikzlibrary{arrows}

\usepackage{import}

\usepackage{mathtools}
\usepackage{mathrsfs}
\usepackage{commath}

\usepackage{framed}

\usetikzlibrary{calc}

\makeatletter
\pgfarrowsdeclare{X}{X}
{
  \pgfutil@tempdima=0.3pt%
  \advance\pgfutil@tempdima by.25\pgflinewidth%
  \pgfutil@tempdimb=5.5\pgfutil@tempdima\advance\pgfutil@tempdimb by.5\pgflinewidth%
  \pgfarrowsleftextend{+-\pgfutil@tempdimb}
  \pgfarrowsrightextend{0pt}
}
{
  \pgfutil@tempdima=0.3pt%
  \advance\pgfutil@tempdima by.25\pgflinewidth%
  \pgfsetdash{}{+0pt}
  \pgfsetroundcap
  \pgfsetmiterjoin
  \pgfpathmoveto{\pgfqpoint{-5.5\pgfutil@tempdima}{-6\pgfutil@tempdima}}
  \pgfpathlineto{\pgfqpoint{5.5\pgfutil@tempdima}{6\pgfutil@tempdima}}
  \pgfpathmoveto{\pgfqpoint{-5.5\pgfutil@tempdima}{6\pgfutil@tempdima}}
  \pgfpathlineto{\pgfqpoint{5.5\pgfutil@tempdima}{-6\pgfutil@tempdima}}
  \pgfusepathqstroke 
}
\makeatother

\newtheorem{theorem}{Theorem}[section]
\newtheorem{lemma}[theorem]{Lemma}

\newtheorem{corollary}[theorem]{Corollary}
\theoremstyle{definition}
\newtheorem{definition}[theorem]{Definition}
\newtheorem*{acknowledgements}{Acknowledgements}

\theoremstyle{remark}
\newtheorem{remark}[theorem]{Remark}

\usepackage{relsize}
\usepackage{exscale}

\usepackage{mathtools}
\usepackage{mathrsfs}

\usepackage{nicefrac}

\usepackage{framed}

\usepackage{datetime}

\usepackage{hyperref} 

\renewcommand{\epsilon}{\varepsilon}

\DeclareMathOperator{\Orb}{Orb}

\title{Orbital shadowing, $\omega$-limit sets and minimality} 
\author{Joel Mitchell}
\date{May 2019}

\begin{document}

\hypersetup{pageanchor=false} 

\begin{abstract} Let $X$ be a compact Hausdorff space, with uniformity $\mathscr{U}$, and let $f \colon X \to X$ be a continuous function. For $D \in \mathscr{U}$, a $D$-pseudo-orbit is a sequence $(x_i)$ for which $(f(x_i),x_{i+1}) \in D$ for all indices $i$. In this paper we show that pseudo-orbits trap $\omega$-limit sets in a neighbourhood of prescribed accuracy after a uniform time period. A consequence of this is a generalisation of a result of Pilyugin \textit{et al}: every system has the second weak shadowing property. By way of further applications we give a characterisation of minimal systems in terms of pseudo-orbits and show that every minimal system exhibits the strong orbital shadowing property.

\end{abstract}

\maketitle
\input{Sections/Introduction.tex}
\input{Sections/Preliminaries.tex}
\input{Sections/Results.tex}

\bibliographystyle{plain} 
\bibliography{bib}

\end{document}

%% file: Sections/Introduction.tex
Let $f \colon X\to X$ be a continuous map on a compact metric space $X$. We say $(X,f)$ is a (discrete) \textit{dynamical system}. A sequence $(x_i)$ in $X$ is called a \emph{$\delta$\textit{-pseudo-orbit}} provided $d(f(x_{i}),x_{i+1}) <\delta$ for each $i$. Pseudo-orbits are clearly relevant when calculating an orbit numerically, as rounding errors mean a computed orbit will in fact be a pseudo-orbit. The (finite or infinite) sequence $(y_i)$ in $X$ is said to \emph{$\epsilon$-shadow} the $(x_i)$ provided $d(y_i,x_i)<\epsilon$ for all indices $i$. First used implicitly by Bowen \cite{Bowen}, a system has \textit{shadowing}, or \textit{the pseudo-orbit tracing property}, if pseudo-orbits are shadowed by true orbits. Since then various other notions of shadowing have been studied, for example, ergodic, thick and Ramsey shadowing \cite{brian-oprocha, bmr, Dastjerdi, Fakhari, Oprocha2016}, limit shadowing \cite{BarwellGoodOprocha, GoodOprochaPuljiz2019,Pilyugin2007}, $s$-limit shadowing \cite{BarwellGoodOprocha,GoodOprochaPuljiz2019, LeeSakai}, orbital shadowing \cite{GoodMeddaugh2016, Pilyugin2007, PiluginRodSakai2002}, and inverse shadowing \cite{CorlessPilyugin, Lee}.  

The \textit{orbital shadowing property} was introduced in \cite{PiluginRodSakai2002} where the authors studied its relationship to classical stability properties, such as structural stability and 
$\Omega$-stability. Informally, a system has orbital shadowing if the closure of the set of points in any pseudo-orbit is close to an orbit closure of a point (see below of precise definitions). Orbital shadowing has since been studied by various authors (e.g \cite{GoodMeddaugh2016, GoodMitchellThomas, Pilyugin2007}). A stronger type of orbital shadowing was introduced in \cite{GoodMeddaugh2016}, aptly named \textit{strong orbital shadowing}, as part of the authors' quest to characterise when the set of $\omega$-limit sets of a system coincides with the set of closed internally chain transitive sets.

In this paper we prove (Theorem \ref{thmOmega}) that every compact metric dynamical system exhibits the following property:
For any $\epsilon>0$ there exist $n \in \mathbb{N}$ and $\delta>0$ such that given any $\delta$-pseudo-orbit 
$(x_i)_{i \geq 0}$ there exists $z\in X$ such that
\[B_\epsilon\left(\{x_i\}_{i=0} ^n \right) \supseteq \omega(z).\]
Thus initial segments of pseudo-orbits trap $\omega$-limit sets in their neighbourhood. As an application of this result we show that compact minimal systems have the strong orbital shadowing property as introduced in \cite{GoodMeddaugh2016}. Our methodology allows us to give a characterisation of minimal systems in terms of pseudo-orbits (Theorem \ref{thmMinimalIFF}). Along the way we generalise a result of Pilyugin \textit{et al} \cite{PiluginRodSakai2002} by showing that every compact Hausdorff system has the second weak shadowing property.

In order to keep our results as general as possible we take the phase space throughout to be compact Hausdorff but not necessarily metric; this is a setting which has attracted an increasing amount of attention in topological dynamics (e.g.  \cite{AuslanderGreschonigNagar,Ceccherini,GoodMacias, GoodMitchellThomas, Hood, MoralesSirvent, YanZeng}). In particular this means all of our results hold in a compact metric setting. 

%% file: Sections/Preliminaries.tex
\section{Preliminaries}

\subsection{Dynamical systems}
A \textit{dynamical system} is a pair $(X,f)$ consisting of a compact Hausdorff space $X$ and a continuous function $f\colon X \to X$. We say that the \textit{orbit} of $x$ under $f$ is the set of points $\{x, f(x), f^2(x), \ldots\}$; we denote this set by $\Orb_f(x)$. For a point $x \in X$, we define the $\omega$\textit{-limit set} of $x$ under $f$, denoted $\omega(x)$, to be the set of limit points of its orbit sequence. Formally
\[\omega(x)= \bigcap_{N \in \mathbb{N}}\overline{\{f^n(x) \mid n >N\}}.\]
Note that as $X$ is compact $\omega(x) \neq \emptyset$ for any $x \in X$ by Cantor's intersection theorem. 

For a dynamical system $(X,f)$, a subset $A \subseteq X$ is said to be \textit{positively invariant} (under $f$) if $f(A) \subseteq A$. The system is \textit{minimal} if there are no proper, nonempty, closed, positively-invariant subsets of $X$. Equivalently, a system is minimal if $\omega(x)=X$ for all $x \in X$. 

If $X$ is a metric space, a sequence $(x_i)_{i \geq 0}$ in $X$ is called a $\delta$-pseudo-orbit if $d(f(x_i), x_{i+1})< \delta$ for all $i \geq 0$.

\begin{definition}
Let $X$ be a metric space. The system $(X,f)$ has the \textit{orbital shadowing} property if for all $\epsilon>0$, there exists $\delta>0$ such that for any $\delta$-pseudo-orbit $( x_i)_{i \geq 0}$, there exists a point $z$ such that 

\[d_H\left(\overline{\{x_i\}_{i\geq 0}}, \overline{\{f^i(z)\}_{i\geq 0}}\right) <\epsilon.\]
\end{definition}
Here $d_H$ denotes the Hausdorff metric, defined on the compact subsets of $X$, which is given by: \[d_H (A,A^\prime)= \inf \{\epsilon>0 \colon A \subseteq B_\epsilon (A^\prime) \text{ and } A^\prime \subseteq B_\epsilon (A)\}. \]
The following weakening of orbital shadowing was introduced in \cite{PiluginRodSakai2002}.
\begin{definition}
Let $X$ be a metric space. The system $(X,f)$ has the \textit{second weak shadowing} property if for all $\epsilon>0$, there exists 
$\delta>0$ such that for any $\delta$-pseudo-orbit $(x_i)_{i\geq 0}$, there exists a point $z$ such that
\[\Orb(z) \subseteq B_\epsilon\left( \{x_i\}_{i\geq 0}\right).\]
\end{definition}
The following strengthening of orbital shadowing was introduced in \cite{GoodMeddaugh2016}. The authors demonstrate it to be distinct.
\begin{definition}
Let $X$ be a metric space. The system $(X,f)$ has the \textit{strong orbital shadowing} property if for all $\epsilon>0$, there exists $\delta>0$ such that for any $\delta$-pseudo-orbit $(x_i)_{i \geq 0}$, there exists a point $z$ such that, for all $N \in \mathbb{N}_0$,

\[d_H\left(\overline{\{x_{N+i}\}_{i\geq 0}}, \overline{\{f^{N+i}(z)\}_{i\geq 0}}\right) <\epsilon.\]
\end{definition}

\subsection{Uniform spaces}
Let $X$ be a nonempty set and $A \subseteq X \times X$. Let $A^{-1}=\{(y,x) \mid (x,y) \in A\}$; we call this the \textit{inverse} of $A$. The set $A$ is said to be \textit{symmetric} if $A=A^{-1}$. For any $A_1, A_2 \subseteq X \times X$ we define the composite $A_1 \circ A_2$ of $A_1$ and $A_2$ as 
\[A_1 \circ A_2 = \{(x,z) \mid \exists y\in X : (x,y) \in A_1, (y,z) \in A_2\}.\]
For any $n \in \mathbb{N}$ and $A \subseteq X \times X$ we denote by $nA$ the $n$-fold composition of $A$ with itself, i.e. 
\[nA=\underbrace{A\circ A\circ A \cdots A}_{n \text{ times}}.\]
The \textit{diagonal} of $X \times X$ is the set $\Delta= \{(x,x) \mid x \in X\}$. A subset $A \subseteq X\times X$ is called an \textit{entourage} if $A \supseteq \Delta$.

\begin{definition} A \textit{uniformity} $\mathscr{U}$ on a set $X$ is a collection of entourages of the diagonal such that the following conditions are satisfied.

\begin{enumerate}[label=\alph*.]
\item $E_1, E_2 \in \mathscr{U} \implies E_1 \cap E_2 \in \mathscr{U}$.
\item $E \in \mathscr{U}, E \subseteq D \implies D \in \mathscr{U}$.
\item $E \in \mathscr{U} \implies D \circ D \subseteq E$ for some $D \in \mathscr{U}$.
\item $E \in \mathscr{U} \implies D^{-1}\subseteq E$ for some $D \in \mathscr{U}$.
\end{enumerate}

\end{definition}

We call the pair $(X, \mathscr{U})$ a {\em uniform space}. We say $\mathscr{U}$ is {\em separating} if $\bigcap_{E \in \mathscr{U}} E = \Delta$; in this case we say $X$ is {\em separated}. A subcollection $\mathscr{V}$ of $\mathscr{U}$ is said to be a {\em base} for $\mathscr{U}$ if for any $E \in\mathscr{U}$ there exists $D \in \mathscr{V}$ such that $E \subseteq D$. Clearly any base $\mathscr{V}$ for a uniformity will have the following properties:
\begin{enumerate}
\item $E_1, E_2 \in \mathscr{U} \implies$ there exists $D \in \mathscr{V}$ such that $D \subseteq E_1 \cap E_2 $.
\item $E \in \mathscr{U} \implies D \circ D \subseteq E$ for some $D \in \mathscr{V}$.
\item $E \in \mathscr{U} \implies D^{-1}\subseteq E$ for some $D \in \mathscr{V}$.
\end{enumerate}
If $\mathscr{U}$ is separating then $\mathscr{V}$ will satisfy $\bigcap_{E \in \mathscr{V}} E = \Delta$.

\begin{remark}\label{RemarkSymFormBase} The symmetric entourages of a uniformity $\mathscr{U}$ form a base for said uniformity. In virtue of this, without loss of generality, we may assume that every entourage in the uniformity that we refer to is symmetric. This will be a standing assumption throughout this paper. 
\end{remark}


For an entourage $E \in \mathscr{U}$ and a point $x \in X$ we define the set $B_E(x)= \{y \in X \mid (x,y) \in E\}$; we refer to this set as the $E$-\textit{ball about} $x$. This naturally extends to a subset $A \subseteq X$; $B_E(A)= \bigcup_{x \in A}B_E(x)$; in this case we refer to the set $B_E(A)$ as the $E$-\textit{ball about} $A$. We emphasise that (see \cite[Section 35.6]{Willard}):
\begin{itemize}
\item For all $x \in X$, the collection $\mathscr{B}_x \coloneqq \{ B_E(x) \mid E \in \mathscr{U} \}$ is a neighbourhood base at $x$, making $X$ a topological space. The same topology is produced if any base $\mathscr{V}$ of $\mathscr{U}$ is used in place of $\mathscr{U}$.
\item The topology is Hausdorff if and only if $\mathscr{U}$ is separating.
\end{itemize}

For a compact Hausdorff space $X$ there is a unique uniformity $\mathscr{U}$ which induces the topology and the space is metric if the uniformity has a countable base (see \cite[Chapter 8]{Engelking}). For a metric space, a natural base for the uniformity would be the $1/2^n$ neighbourhoods of the diagonal. 


We may use uniformities to give appropriate definitions of orbital shadowing, second weak shadowing and strong orbital shadowing in the more general setting of uniform spaces. First of all, given an entourage $D \in \mathscr{U}$, a sequence $(x_i)_{i \geq 0}$ in $X$ is called a $D$-pseudo-orbit if $(f(x_i), x_{i+1}) \in D$ for all $i \geq 0$.

\begin{definition}
Let $X$ be a uniform space. The system $(X,f)$ has the \textit{orbital shadowing} property if for all $E \in \mathscr{U}$, there exists $D \in \mathscr{U}$ such that for any $D$-pseudo-orbit $( x_i)_{i \geq 0}$, there exists a point $z \in X$ such that 
\[\Orb(z) \subseteq B_E\left( \{x_i\}_{i\geq 0}\right)\]
and
\[\{x_i\}_{i\geq 0} \subseteq B_E\left(\Orb(z) \right).\]
\end{definition}

\begin{definition}
Let $X$ be a uniform space. The system $(X,f)$ has \textit{second weak shadowing} if for all $E \in \mathscr{U}$, there exists 
$D \in \mathscr{U}$ such that for any $D$-pseudo-orbit $(x_i)_{i\geq 0}$, there exists a point $z$ such that

\[\Orb(z) \subseteq B_E\left( \{x_i\}_{i\geq 0}\right).\]
\end{definition}

\begin{definition}
Let $X$ be a uniform space. The system $(X,f)$ has the \textit{strong orbital shadowing} property if for all $E \in \mathscr{U}$, there exists $D \in \mathscr{U}$ such that for any $D$-pseudo-orbit $( x_i)_{i \geq 0}$, there exists a point $z \in X$ such that, for all $N \in \mathbb{N}_0$,
\[\{f^{N+i}(z)\}_{i\geq 0}\subseteq B_E\left( \{x_{N+i}\}_{i\geq 0}\right)\]
and
\[\{x_{N+i}\}_{i\geq 0} \subseteq B_E\left(\{f^{N+i}(z)\}_{i\geq 0} \right).\]
\end{definition}
When $X$ is a compact metric space these definitions coincide with the previously given metric versions.

\medskip

Throughout this paper, as $X$ is a compact Hausdorff space, we denote the unique uniformity associated with $X$ by $\mathscr{U}$. 

%% file: Sections/Results.tex
\section{Main results}

\begin{lemma}\label{Every}
Let $(X,f)$ be a dynamical system where $X$ is a compact Hausdorff space. Then $(X,f)$ satisfies the following:
\[\forall E \in \mathscr{U} \, \forall x \in X \exists n \in \mathbb{N} \exists z \in X \text{ s.t. } \]
\[ \bigcup_{i=1} ^n B_{E}\left(f^i(x)\right) \supseteq \omega(z).\]

\end{lemma}

\begin{proof} Take $E \in \mathscr{U}$ and pick $x \in X$. Let $E_0 \in \mathscr{U}$ be such that $2E_0 \subseteq E$. Take a finite subcover of the open cover $\{B_{E_0}(y) \mid y \in \omega(x)\}$ of $\omega(x)$. For each element of this subcover there exists $n$ such that $f^n(x)$ lies inside it. Pick one such $n$ for each element and then take the largest. The result follows.
\end{proof}

\begin{lemma}\label{Every2}
Let $(X,f)$ be a dynamical system where $X$ is a compact Hausdorff space. Then $(X,f)$ satisfies the following:
\[\forall E \in \mathscr{U} \, \exists n \in \mathbb{N} \text{ s.t. } \forall x \in X \exists z \in X \text{ s.t. } \]
\[ \bigcup_{i=1} ^n B_{E}\left(f^i(x)\right) \supseteq \omega(z).\]
\end{lemma}

\begin{proof}
Fix $E \in \mathscr{U}$. Let $E_0 \in \mathscr{U}$ be such that $2E_0 \subseteq E$. For each $x \in X$ let $n_x \in \mathbb{N}$ be as in the condition in Lemma \ref{Every} for $E_0$ and let $D_x \in \mathscr{U}$ be such that, for any $y \in X$, if $(x,y) \in D_x$ then, for each $i \in \{0,\ldots, n_x\}$, $(f^i(x), f^i(y)) \in E_0$. The collection $\{B_{D_x}(x) \mid x \in X\}$ forms an open cover. Let
\[\left\{B_{D_{x_i}}(x_i)\mid i \in \{1,\ldots, k\}\right\},\]
be a finite subcover. Take $n = \max_{i \in \{1, \ldots, k\}} n_{x_i}$. Then, by composition, for any $x \in X$ there exists $z \in X$ such that 
\[\bigcup _{i=1} ^n B_E
\left(f^i(x)\right) \supseteq \omega(z). \]
\end{proof}

\begin{theorem}\label{thmOmega}
Let $(X,f)$ be a dynamical system where $X$ is a compact Hausdorff space. Then for any $E \in \mathscr{U}$ 
there exist $n \in \mathbb{N}$ and $D \in \mathscr{U}$ such that given any $D$-pseudo-orbit 
$(x_i)_{i \geq 0}$ there exists $z\in X$ such that
\[B_E\left(\{x_i\}_{i=0} ^n \right) \supseteq \omega(z).\]
In particular,
\[B_E\left(\{x_i\}_{i \geq 0}\right) \supseteq \omega(z).\]
\end{theorem}

\begin{proof}
 Let $E \in \mathscr{U}$ be given and let $E_0 \in \mathscr{U}$ be such that $2E_0 \subseteq E$. Take $n \in \mathbb{N}$ as in the condition in Lemma \ref{Every2} with respect to $E_0$. By uniform continuity we can choose $D \in \mathscr{U}$ such that every $D$-pseudo-orbit $E_0$-shadows the first $n$ iterates of its origin. Explicitly: Let $D_1 \subseteq E_0$ be an entourage such that, for any $y, z \in X$, if $(y,z) \in D_1$ then $(f(y),f(z)) \in E_0$. For each $i \in \{2,\ldots, n\}$ let $D_i \in \mathscr{U}$ be such that $2D_i \subseteq f^{-1}(D_{i-1}) \cap D_{i-1}$.

Now take $D\coloneqq D_n$. Suppose $(x_i)_{i \geq 0}$ is a $D$-pseudo-orbit. Then $(f^i(x_0),x_i) \in E_0$ for all $i \in \{0, \ldots n\}$. By the given condition there exists $z \in X$ such that 
\[ \bigcup_{i=1} ^n B_{E_0}\left(f^i(x_0)\right) \supseteq \omega(z).\]
Since, for each $i \in \{0, \ldots, n\}$, $(f^i(x_0),x_i) \in E_0$ it follows from entourage composition, and the fact that $2E_0 \subseteq E$, that 
\[ B_E\left( \{x_i\}_{i=0} ^n\right) \supseteq \omega(z).\]
\end{proof}

The fact that all compact Hausdorff systems exhibit second weak shadowing now follows as a simple corollary to Theorem \ref{thmOmega}. Note that Corollary \ref{CorOrbTrap} is a generalisation of \cite[Theorem 3.1]{PiluginRodSakai2002}.

\begin{corollary}\label{CorOrbTrap}
Let $(X,f)$ be a dynamical system where $X$ is a compact Hausdorff space. Then the system has second weak shadowing.
\end{corollary}
\begin{proof}
Let $E \in \mathscr{U}$ be given and let $D \in \mathscr{U}$ correspond to this as in Theorem \ref{thmOmega}. Take a $D$-pseudo-orbit $(x_i)_{i \geq 0}$. By Theorem \ref{thmOmega} there exists $z \in X$ such that 
\[ B_E\left( \{x_i\}_{i\geq 0}\right) \supseteq \omega(z).\]
Since $\omega$-limit sets are positively invariant it follows that for any $y \in \omega(z)$
\[\Orb(y) \subseteq B_E\left( \{x_i\}_{i\geq 0}\right).\]
It remains to note that $\omega(z) \neq \emptyset$ as $X$ is compact.
\end{proof}

\begin{theorem}\label{thmMinimalIFF}
Let $X$ is a compact Hausdorff space and $f \colon X \to X$ be a continuous function. The system $(X,f)$ is minimal if and only if for any $E \in \mathscr{U}$ there exist $D \in \mathscr{U}$ and $n \in \mathbb{N}$ such that for any two $D$-pseudo-orbits $(x_i)_{i \geq 0}$ and $(y_i)_{i \geq 0}$ 
\[\{y_i\}_{i =0} ^n \subseteq B_E\left( \{x_i\}_{i=0} ^n\right)\]
and
\[\{x_i\}_{i =0} ^n \subseteq B_E\left( \{y_i\}_{i=0} ^n\right).\]
\end{theorem}

\begin{proof}
First suppose the system is minimal. Let $E \in \mathscr{U}$ be given. Take $D \in \mathscr{U}$ and $n \in \mathbb{N}$ corresponding to $E$ as in Theorem \ref{thmOmega}. Now let $(x_i)_{i \geq 0}$ and $(y_i)_{i \geq 0}$ be two $D$-pseudo-orbits. By Theorem \ref{thmOmega} there exist $z_1, z_2 \in X$ such that
$B_E\left(\{x_i\}^n _{i=0} \right) \supseteq \omega(z_1)$
and
$B_E\left(\{y_i\}^n _{i=0} \right) \supseteq \omega(z_2).$
As $(X,f)$ is minimal $\omega(z_1)=\omega(z_2)=X$. It follows that
$B_E\left(\{x_i\}^n _{i=0} \right)=B_E\left(\{y_i\}^n _{i=0} \right)=X.$
Hence
\[\{y_i\}_{i =0} ^n \subseteq B_E\left( \{x_i\}_{i=0} ^n\right)\]
and
\[\{x_i\}_{i =0} ^n \subseteq B_E\left( \{y_i\}_{i=0} ^n\right).\]

Now suppose the system is not minimal. Then there exists $x \in X$ such that $\omega(x) \neq X$. Pick $y \in \omega(x)$ and let $z \in X \setminus \omega(x)$. Take $E \in \mathscr{U}$ such that $B_E(z) \cap \omega(x) = \emptyset$. As $E$ is symmetric by our standing assumption, $z \notin B_E(\omega(x))$. Consider the pseudo-orbits given by the orbit sequences of $y$ and $z$: these are $D$-pseudo-orbits for any $D \in \mathscr{U}$. As $\omega$-limit sets are positively invariant, $\Orb(y) \subseteq \omega(x)$. Since $z \notin B_E(\omega(x))$ it also follows that $z \notin B_E(\Orb(y))$. In particular $\Orb(z) \not\subseteq B_E\left(\Orb(y)\right)$.
\end{proof}

For the case when $X$ is a compact metric space Theorem \ref{thmMinimalIFF} may be formulated as follows: A dynamical system $(X,f)$ is minimal precisely when for any $\epsilon>0$ there exist $\delta>0$ and $n \in \mathbb{N}$ such that for any two $\delta$-pseudo-orbits $(x_i)_{i \geq 0}$ and $(y_i)_{i \geq 0}$ 
\[d_H(\{x_i\}_{i =0} ^n , \{y_i\}_{i=0} ^n) <\epsilon.\]

\begin{corollary}
Let $X$ is a compact Hausdorff space and $f \colon X \to X$ be a continuous function. The system $(X,f)$ is minimal if and only if for any $E \in \mathscr{U}$ there exist $D \in \mathscr{U}$ and $n \in \mathbb{N}$ such that for any $D$-pseudo-orbit $(x_i)_{i \geq 0}$ we have $B_E\left(\{x_i\}_{i=0} ^n \right)=X$.
\end{corollary}
\begin{proof}
Immediate from the proof of Theorem \ref{thmMinimalIFF}.
\end{proof}

\begin{corollary}
Let $X$ be a compact Hausdorff space. If $(X,f)$ is a minimal dynamical system then it exhibits the strong orbital shadowing property.
\end{corollary}

\begin{proof}
Let $E \in \mathscr{U}$ be given. Take $D \in \mathscr{U}$ and $n \in \mathbb{N}$ corresponding to $E$ as in Theorem \ref{thmMinimalIFF}. Now let $(x_i)_{i \geq 0}$ be a $D$-pseudo-orbit and pick any $z \in X$. Since $(x_{N+i})_{i \geq 0}$ and $(f^{N+i}(z))_{i \geq 0}$ are $D$-pseudo-orbits for all $N \in \mathbb{N}_0$, by Theorem \ref{thmMinimalIFF},
\[\{f^{N+i}(z)\}_{i =0} ^n \subseteq B_E\left( \{x_i\}_{i=0} ^n\right)\]
and
\[\{x_{N+i}\}_{i =0} ^n \subseteq B_E\left( \{f^{N+i}(z)\}_{i=0} ^n\right).\]

\end{proof}

\begin{acknowledgements} The funding provided by EPSRC/University of Birmingham is
gratefully acknowledged. The author would also like to thank Chris Good for his support and guidance.
\end{acknowledgements}